\documentclass[10pt]{article}
\usepackage[T1]{fontenc}
\usepackage[cp1250]{inputenc}
\usepackage{lmodern}

\usepackage[english]{babel}
\usepackage{algorithm}
\usepackage{algpseudocode}
\usepackage{latexsym}
\usepackage{amsmath}
\usepackage{indentfirst}
\usepackage{amsthm}
\usepackage{amssymb}
\usepackage{amsfonts}
\usepackage{stackrel}
\usepackage{dsfont}
\usepackage{tikz}
\usepackage{slashbox}
\usepackage[mathscr]{eucal}
\usepackage{authblk}
\usepackage{hyperref}
\usepackage{mathabx}
\usepackage{bookmark}
\usepackage{xcolor}

\newtheorem{thm}{Theorem}

\newcommand{\naturals}{\mathbb{N}}

\newcommand{\eps}{\varepsilon}
\newcommand{\pistat}{\pi^*}

\begin{document}
\title{Darwinian evolution in Malthusian population growth model and Markov chains}
\author{Mateusz Krukowski}
\affil{Institute of Mathematics, \L\'od\'z University of Technology, \\ W\'ol\-cza\'n\-ska 215, \
90-924 \ \L\'od\'z, \ Poland \\ \vspace{0.3cm} e-mail: mateusz.krukowski@p.lodz.pl}
\maketitle

\begin{abstract}
The paper is devoted to the study of Darwinian evolution in two mathematical models. The first one is a variation on the Malthusian population growth model with Verhulst's environmental capacity. The second model is grounded in the theory of Markov chains and their stationary distributions. We prove preliminary results regarding both models and pose conjectures, which are supported by computer simulations.    
\end{abstract}

\vspace{0.2cm}
% 92D15 -- Problems related to evolution
% 92D25 -- Population dynamics (general) 

\section{Introduction}
\label{section:Introduction}

% historical background
The observation that crossbreeding of animals and plants could favour certain desirable traits had been known to breeders and farmers for millenia. Arguably the most familiar example is the grey wolf (Canis lupus), whose domestication resulted in over two hundred breeds of dog (Canis lupus familiaris). Another case in point is the wild cabbage (Brassica oleracea), which the horticulturalists have ``chiseled into'' vegetables such as broccoli, Brussels sprouts, cauliflower, kale, kohlrabi, romanescu etc. It is quite remarkable that these instances did not develop into a comprehensive theory until the dawn of XIX century. 

% summary of evolution theory
Over hundred and fifty years prior to Darwin setting foot on board HMS Beagle, the scientific revolution began to transform people's perspective on the world around us. Giant leaps in understanding the laws of nature were not restricted to physics or chemistry, but pertained to biology in equal measure. Arguably, every comprehensive list of the brightest stars of the era should include:
\begin{itemize}
	\item Georges--Louis Leclerc, Comte de Buffon (1707--1788), who devoted over 40 years to compile a monumental 36--volume treatise \textit{Histoire Naturelle},
	\item Erasmus Darwin (1731--1802), Charles' grandfather, who anticipated that ``all warm--blooded animals have arisen from one living filament'', (see Section XXXIX ``Of Generation'' in \cite{ErasmusDarwin})
	\item Jean-Baptiste de Lamarck (1744--1829), a father of transmutation theory and first evolutionary framework (see \cite{Lamarck}).
\end{itemize}

Charles Darwin (1809--1882), studied at the University of Edinburgh, where he came across the ideas of Lamarck. After moving to Cambridge, he read William Paley's \textit{Natural Theology} (see \cite{Paley}) and Alexander Humboldt's \textit{Personal Narrative}, which might have instigated dreams of exotic adventures. An opportunity came knocking on Darwin's door, when one of his professors, John Henslow, recommended him for a voyage organised by viceadmiral Robert FitzRoy. A five year long expedition around the globe has certainly made a lasting impression on the young Englishman, straight out of university. Witnessing slavery in Brasil or an earthquake in Chile must have had an impact on Darwin, but it were his encounters with exotic fauna on isolated islands (like finches in Galapagos Archipelago) that eventually immortalized his name. 

HMS Beagle concluded its journey in 1836, but Darwin continued to perfect his work for over 20 years, publishing \textit{On the origins of species} in 1859 (see \cite{Darwin}). At its core, his theory of evolution rests on two pillars:
\begin{itemize}
	\item imperfect replication, which allows for tiny mutations in every new generation, and
	\item natural selection, which is an external force exerted on species eliminating weak individuals.
\end{itemize}

\noindent
Both pillars go hand in glove and can be thought of two sides of the same coin. Imperfect replication accounts for diversity and exploration of multitude of options, which are later trimmed down by natural selection. Individuals with weak variants of the mutation swiftly perish, leaving only the fittest to prosper, i.e., reproduce. The survival and passing on the genes are reserved only for the best adapted.     

% summary of the paper
The current paper is focused on studying Darwinian ideas through the lens of two well--known mathematical domains: population growth models and Markov chains. Section \ref{Section:PearlVerhulst} introduces Malthusian model of population dynamics and proceeds with a discussion of Verhulst's environmental capacity. We build on these models by adding two pillars of Darwinian evolution. Function $f_{mut}$ is designed to account for imperfect replication, while $f_{mort}$ represents the invisible hand of natural selection. The goal of the section is to provide ample evidence (both in the form of mathematical theorems as well as computer simulations), that the introduction of the two pillars results in Darwinian evolution.   

Section \ref{Section:Markov} shifts our focus to Markov chains, but the overal aim remains unaltered. We commence by reviewing the fundamentals of Markov chain theory with special emphasis on stationary distributions. Next, we introduce the expected average velocity as a measure of Darwinian evolution and establish its properties under particular form of the transition matrix. We conclude with a series of open questions regarding possible generalizations of our results. These conjectures are motivated by a number of computer simulations, also described in the paper.

\section{Evolution in Verhulst model of population growth}
\label{Section:PearlVerhulst}

% Malthusian model 
In 1798, English economist and cleric Thomas Robert Malthus (1766--1834) published ``\textit{An Essay on the Principle of Population}'' (see \cite{Malthus}). In the first chapter, Malthus expressed his concern regarding the threat of overpopulation: ``\textit{(...) the power of population is indefinitely greater than the power in the earth to produce subsistence for man. Population, when unchecked, increases in a geometrical ratio. Subsistence increases only in an arithmetical ratio. A slight acquaintance with numbers will shew the immensity of the first power in comparison of the second.}'' Malthus went on to provide the following example in the second chapter of his treatise: ``\textit{Taking the population of the world at any number, a thousand millions, for instance, the human species would increase in the ratio of -- 1, 2, 4, 8, 16, 32, 64, 128, 256, 512, etc.}'' The underlying model of Malthus' reasoning takes the form $P_{n+1} = 2P_n,$ where $P_n$ stands for the size of (human) population at year $n$. Surprisingly, Malthus's intuition is, to some extent, corroborated by empirical data. In 1999, United Nations Department of Economic and Social Affairs issued a booklet \textit{The World at Six Billion} containing the estimates of the world population within the last millennium:
\begin{center}
\begin{tabular}{ |c|c|c|c|c|c|c|c|c } 
 \hline
 year & 1000 & 1250 & 1500 & 1750 & 1800 & 1850 & 1900 & 1910  \\ \hline
 population (in billions) & 0.31 & 0.40 & 0.50 & 0.79 & 0.98 & 1.26 & 1.65 & 1.75 \\ \hline
\end{tabular}
\end{center}

\begin{center}
\begin{tabular}{ c|c|c|c|c|c|c|c|c|} 
\hline
1920 & 1930 & 1940 & 1950 & 1960 & 1970 & 1980 & 1990 & 1999 \\ \hline
1.86 & 2.07 & 2.30 & 2.52 & 3.02 & 3.70 & 4.04 & 5.27 & 5.981 \\ \hline
\end{tabular}
\end{center}

\noindent
Plotting the data (see Figure \ref{worldpopulation}), we see that the world population follows an exponential trajectory, as forseen by the author of the essay. 

% Malthusian model cannot apply indefinitely
Successful as Malthusian model is, we would be gullible to believe that the exponential trend can continue indefinitely. If that were the case, we would eventually run out of room for new people to be born and the Earth would be too crammed even to move. Thus, demographic experts believe that we are steadily approaching the point when the exponential trend will be broken and the exponential Malthusian model $P_{n+1} = 2P_n$ will no longer apply. The world population is estimated to saturate around 11 billion people by the end of the 21st century (see Chapter 3 in \cite{Rosling}). Thus, the Malthusian model of indefinite exponential growth needs a major revision.
\begin{figure}
\centering
\includegraphics[width=0.8\textwidth]{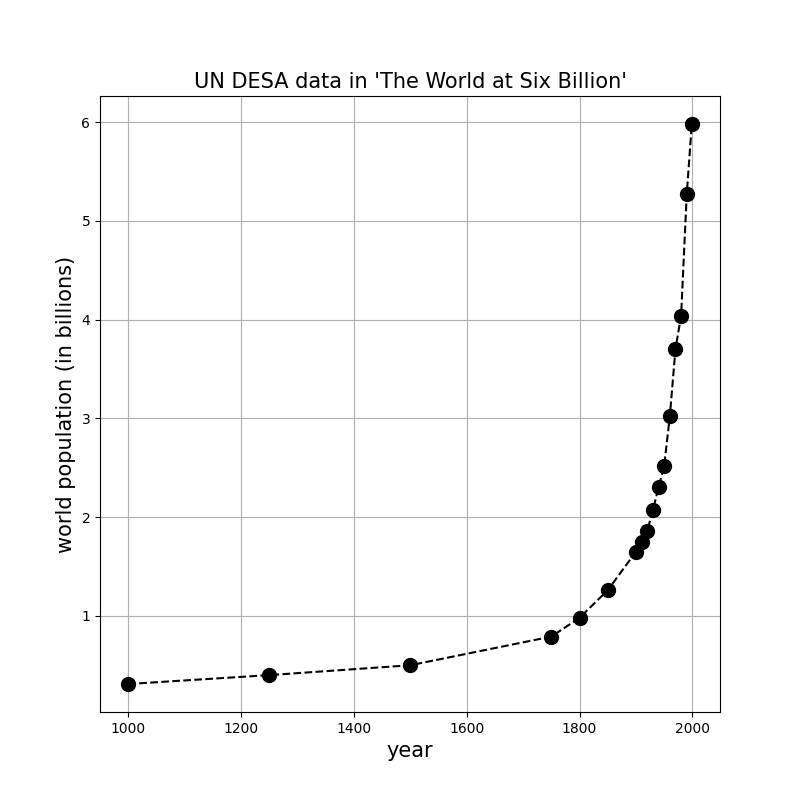}
\caption{World population according to UN DESA booklet ``The World at Six Billion''}
\label{worldpopulation}
\end{figure}

Arguably the most popular modification of Malthusian model dates back to roughly 40 years after ``\textit{An Essay on ...}'', when Pierre Francois Verhulst (see \cite{Verhulst}) argued that even if ``\textit{one has doubled the yield of the soil in the first twenty-five years, one will have scarcely made it produce perhaps a third more in the second period. The virtual growth of the population thus finds a limit in the area and fertility of the country, and the population will consequently tend to become more and more stationary. (...) All of the formulas by which one attempts to represent the law of population must then satisfy the condition that they admit a maximum that is attained only after an infinitely-extended epoch. That maximum will be the population count after it has become stationary.}'' Verhulst investigated the differential model
\begin{gather}
\frac{dP}{dt} = rP\left(1 - \frac{P}{K}\right),
\label{Verhulstequation}
\end{gather}

\noindent
where $r$ is the population growth rate and $K$ is the environmental capacity. He applied the formula to study the population of France, Belgium and the county of Essex between 1811 and 1831. 

Since we will be concerned with computer simulations shortly, we prefer to use a discrete version of Verhulst's model, namely
\begin{gather}
P_{n+1} = (1-m)P_n + bP_n\left(1 - \frac{P_n}{K}\right)\mathds{1}_{[0,K]}(P_n),
\label{discreteVerhulstwithmortality}
\end{gather}

\noindent
in which we have replaced the growth rate $r$ with mortality rate $m\in (0,1]$ and birth rate $b > 0.$ The term $(1-m)P_n$ means that death takes it toll on the old generation $P_n$, i.e., not everyone lives to become part of the new generation $P_{n+1}.$ On the other hand, $bP_n$ models the number of babies the generation $P_n$ produces. The population keeps reproducing and eventually hits the environmental capacity $K,$ in which case people are too busy fighting for survival and they don't multiply at all, hence the term $\left(1 - \frac{P_n}{K}\right)\mathds{1}_{[0,K]}(P_n).$ Let us observe that $\mathds{1}_{[0,K]}(P_n)$ ensures that the number of delivered babies is never negative. Without it, we would allow for ``negative births'' if $P_n$ exceeded the capacity $K.$ 

Verhulst's discrete model \eqref{discreteVerhulstwithmortality} assumes that the members of the population are indistinguishable. The individuals are identitcal, so the only evolution we can witness is the evolution of the population as a whole. Interesting as it is in and of itself, we proceed with introducing a ``spice'' to our model, i.e., a single trait which divides the whole population $P_n$ into clusters. Our goal is to demonstrate that Darwinian evolution is a phenomenon emerging from Verhulst's discrete model, once we account for its driving forces: mutations amongst offspring and selection mechanism. 

Let the function $P_n: V \longrightarrow [0,\infty)$ describe the population with $V$ standing for possible values of the the trait and $P_n(v)$ standing for the number of individuals in generation $n$ with the trait at level $v$. We find it convenient to think of the population as rabbits (rather than humans like Malthus or Verhulst) and the trait as their average velocity. Adopting this terminology, $P_n(v)$ is the number of rabbits with average velocity $v.$ The units are unimportant for our deliberations, but we might as well stick to km$/$h. Our model presupposes that $v$ cannot exceed some maximal value $v_{max}$, since there is a physical limit to how fast an animal can run.  

% first pillar - mutations
We assume that every individual with trait $v$ can give birth to an identical individual with probability $f_{mut}(0),$ but can also give birth to individuals with trait $v+1$ with probability $f_{mut}(1)$ or trait $v-1$ with probability $f_{mut}(-1)$ etc. This means that the old generation $P_n$ gives birth to 
\begin{gather}
bP_n\star f_{mut}(v) := b\sum_{u=1}^{v_{max}}\ P_n(u)f_{mut}(v-u)
\label{mutationconvolution}
\end{gather}

\noindent
rabbits with trait $v$. We make the following (hopefully quite reasonable) assumptions regarding the mutation function:
\begin{enumerate}
	\item $f_{mut}:\{-v_{max}, \ldots, 0, \ldots, v_{max}\} \longrightarrow [0,1],$ which is a necessary condition for the convolution \eqref{mutationconvolution} to be well-defined.
	\item $f_{mut}(-u) = f_{mut}(u)$ for every $u.$ This means that both ``up'' and ``down'' mutations are equally likely. The mutation function itself does not  presuppose whether it is better to be an individual with trait $v+1$ or $v-1.$ It does not discriminate whether it is better to run faster or slower, such a preference is contingent upon the mortality function. 
	\item $u\mapsto f_{mut}(u)$ is decreasing for nonnegative $u$. This means that if the parent-rabbits have trait $v$, then it is most likely that the offspring will also have trait $v$. It is far less likely that the offspring will inherit traits $v\pm 1,$ not to mention traits $v\pm 2$ etc.
	\item We have 
	$$\|f_{mut}\|_1 := \sum_{u=-v_{max}}^{v_{max}}\ f_{mut}(u) = 1,$$
	
	\noindent
	which simply states that $f_{mut}$ is a probability distribution. 
\end{enumerate}

% second pillar - mortality rate
As for the second pillar of Darwinian evolution we introduce a selection mechanism in the form of mortality rate function $f_{mort}: V \longrightarrow (0,1).$ The function accounts for the fact that specimen with different average velocities fare differently in the environment -- some perish shortly after birth, others thrive and manage to pass on their genes. Who lives and who dies is dictated by the mortality function, which is a dependent on the trait $v$.

We are now ready to put forward our model of population dynamics with a variable trait, i.e.,   
\begin{equation}
P_{n+1}(v) = (1 - f_{mort}(v))P_n(v) + bP_n \star f_{mut}(v) \left(1 - \frac{\|P_n\|_1}{K}\right)\mathds{1}_{[0,K]}(\|P_n\|_1)
\label{model}
\end{equation}

\noindent
for every $v = 1,\ldots,v_{max}$, where 
$$\|P_n\|_1 := \sum_{v = 1}^{v_{max}}\ P_n(v)$$ 

\noindent
We believe it bears strong resemblance to Verhulst's discrete equation \eqref{discreteVerhulstwithmortality}, from which it originated.  

One of the main properties of the model, which is a result of environmental capacity, is the boundedness of the population. In order to prove this fact rigorously, let 
$$F_{mut}(u) : = \sum_{v = 1}^{v_{max}}\ f_{mut}(v-u)$$

\noindent
for every $u=1,\ldots,v_{max}$ and let $\|\cdot\|_1, \|\cdot\|_2$ and $\|\cdot\|_{\infty}$ be the $\ell_1, \ell_2$ and maximum norm, respectively. 

\begin{thm}
The population is bounded at all times, i.e., there exists an upper bound 
\begin{gather}
B_{up} := \max\left(\frac{(\|1-f_{mut}\|_2 + b\|F_{mut}\|_{\infty})^2}{4b\|F_{mut}\|_{\infty}}\cdot K, K, \|P_0\|_1\right)
\label{Bup}
\end{gather}

\noindent
such that $\|P_n\|_1 \leqslant B_{up}$ for all $n\in\naturals_0.$
\end{thm}
\begin{proof}
Firstly, we observe that 
$$\|F_{mut}\|_{\infty} := \max_u\ F_{mut}(u) \leqslant 1$$ 

\noindent
due to $f_{mut}$ being a probability distribution. Furthermore, applying $\|\cdot\|_1$ to \eqref{model} we have 
\begin{equation}
\begin{split}
\|P_{n+1}\|_1 &\leqslant \|(1- f_{mort}) P_n\|_1 + \|b P_n \star f_{mut}\|_1 \left(1 - \frac{\|P_n\|_1}{K}\right)\mathds{1}_{[0,K]} \\
&\leqslant \|1-f_{mort}\|_{\infty} \|P_n\|_1 + b \sum_{v = 1}^{v_{max}}\sum_{u = 1}^{v_{max}}\ P_n(u) f_{mut}(v - u) \left(1 - \frac{\|P_n\|_1}{K}\right)\mathds{1}_{[0,K]}(\|P_n\|_1)\\
&\leqslant \bigg( \|1-f_{mort}\|_{\infty} + b \|F_{mut}\|_{\infty} \left(1 - \frac{\|P_n\|_1}{K}\right)\mathds{1}_{[0,K]}(\|P_n\|_1) \bigg) \|P_n\|_1
\end{split}
\label{estimateofP}
\end{equation}

\noindent
We define $B_{up}$ as in \eqref{Bup} and consider two cases: 
\begin{description}
	\item[\hspace{0.4cm} Case 1.] If $\|P_n\|_1 \leqslant K$ then   
	$$\|P_{n+1}\|_1 \stackrel{\eqref{estimateofP}}{\leqslant} \bigg( \|1-f_{mort}\|_{\infty} + b \|F_{mut}\|_{\infty} \left(1 - \frac{\|P_n\|_1}{K}\right) \bigg) \|P_n\|_1.$$
	
	\noindent
	The expression on the right hand side is quadratic in $\|P_n\|_1$ and attains its maximum for 
	$$\frac{\|1-f_{mut}\|_{\infty} + b\|F_{mut}\|_{\infty}}{2b\|F_{mut}\|_{\infty}}\cdot K.$$
	
	\noindent
	Hence 
	\begin{equation*}
	\begin{split}
	\|P_{n+1}\| &\leqslant \bigg(\|1-f_{mort}\|_{\infty} + b \|F_{mut}\|_{\infty} \left(1 - \frac{\|1-f_{mut}\|_{\infty} + b\|F_{mut}\|_{\infty}}{2b\|F_{mut}\|_{\infty}}\right) \bigg)\cdot \frac{(\|1-f_{mut}\|_{\infty} + b\|F_{mut}\|_{\infty})}{2b\|F_{mut}\|_{\infty}}\cdot K\\
	&=\frac{(\|1-f_{mut}\|_{\infty} + b\|F_{mut}\|_{\infty})^2}{4b\|F_{mut}\|_{\infty}}\cdot K \leqslant B_{up}.
	\end{split}
	\end{equation*}
	
	\item[\hspace{0.4cm} Case 2.] If $\|P_n\|_1 \in (K, B_{up}]$ then 
	$$\|P_{n+1}\|_1 \stackrel{\eqref{estimateofP}}{\leqslant} \|1-f_{mort}\|_{\infty} \|P_n\|_1 \leqslant B_{up},$$
	
	\noindent
	since $\|1-f_{mort}\|_{\infty} < 1.$
\end{description}

The two separate cases prove that if $\|P_n\|_1 \leqslant B_{up}$ then $\|P_{n+1}\|\leqslant B_{up}.$ Furthermore, $\|P_0\|\leqslant B_{up}$ and by the principle of mathematical induction we are done.  
\end{proof}

At this stage, we know that the population of rabbits will never exceed a certain upper bound. Next, we want to investigate the situation when the whole population goes extinct. Although we do not have a full characterization, we are able to provide a sufficient condition. As one might expect, the complete annihilation of the population occurs when the mortality rate is ``too big'' compared to the birth rate. The following result fills in the mathematical details of this claim:  

\begin{thm}
If 
\begin{gather}
\|1 - f_{mort}\|_{\infty} + b\|F_{mut}\|_{\infty} \leqslant 1
\label{extinctioncond}
\end{gather} 

\noindent
then $\lim_{n\rightarrow\infty}\ \|P_n\|_1=0$.
\end{thm}
\begin{proof}
Without loss of generality, we may assume that $\|P_0\|_1 < K$, since if $\|P_0\|_1 \geqslant K$ then there exists $N\in\naturals_0$ such that $\|P_N\|_1 \in [0,K).$
Next, an inductive reasoning reveals that 
\begin{equation*}
\|P_{n+1}\|_1 \stackrel{\eqref{estimateofP}}{\leqslant} \bigg(\|1-f_{mort}\|_{\infty}  + b \|F_{mut}\|_{\infty} \left(1 - \frac{\|P_n\|_1}{K}\right) \bigg) \|P_n\|_1\leqslant \|P_n\|_1
\end{equation*}

\noindent
for every $n\geqslant N,$ which means that the sequence $(\|P_n\|_1)$ is eventually nonincreasing. As such, it has a limit $\ell$, which satisfies 
$$\ell \leqslant \bigg(\|1-f_{mort}\|_{\infty}  + b \|F_{mut}\|_{\infty} \left(1 - \frac{\ell}{K}\right) \bigg) \ell.$$ 

\noindent
This can happen only if $\ell=0,$ which concludes the proof.
\end{proof}

In order to witness the evolution of the population, we assume that the mortality function $f_{mort}$ is decreasing. For the sake of simplicity of our model, we put $f_{mort}(v) = \frac{1}{v}$ and 
\begin{gather}
f_{mut}(k) := \left\{\begin{array}{cl}
0.9 & \text{for }k=0,\\
0.05& \text{for }k=\pm 1,\\
0 & \text{otherwise.}
\end{array}\right.
\label{fmutconcrete}
\end{gather}

\begin{figure}
\centering
\includegraphics[width=\textwidth]{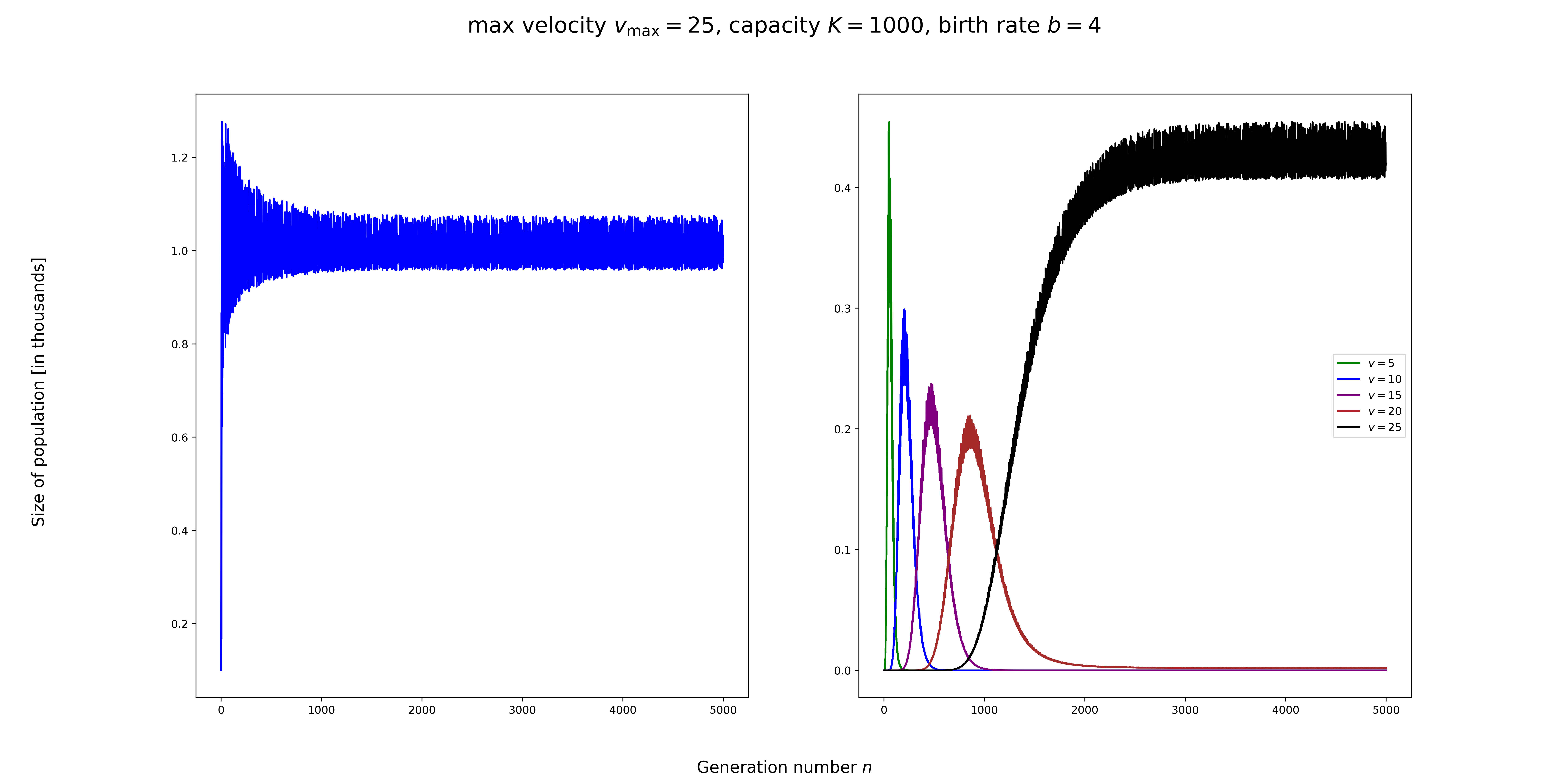}
\caption{Population dynamics in model \eqref{model}}
\label{5velocities}
\end{figure}

\noindent
Starting with the initial fluffle of size 100 with $v=1$, the left plot in Figure \ref{5velocities} depicts the evolution of the whole rabbit population in time. The right plot in Figure \ref{5velocities} focuses on particular ``sections'' of this population at five distinct velocities, namely $v = 5$ (green), $10$ (blue), $15$ (purple), $20$ (brown), $25$ (black). We infer that specimens with low average velocity (green and blue plots) peak at the beginning of the simulation and then die off. Rabbits with middle average velocity (purple and brown plots) last a little longer, but eventually share the fate of their predecessors. They are driven to extinction by individuals, which are ``fitter for survival'', i.e., move faster (black plot). The distribution of the population according to velocity (after $5000$ generations) is depicted in Figure \ref{fighistogram}.
\begin{figure}
\centering
\includegraphics[width=0.8\textwidth]{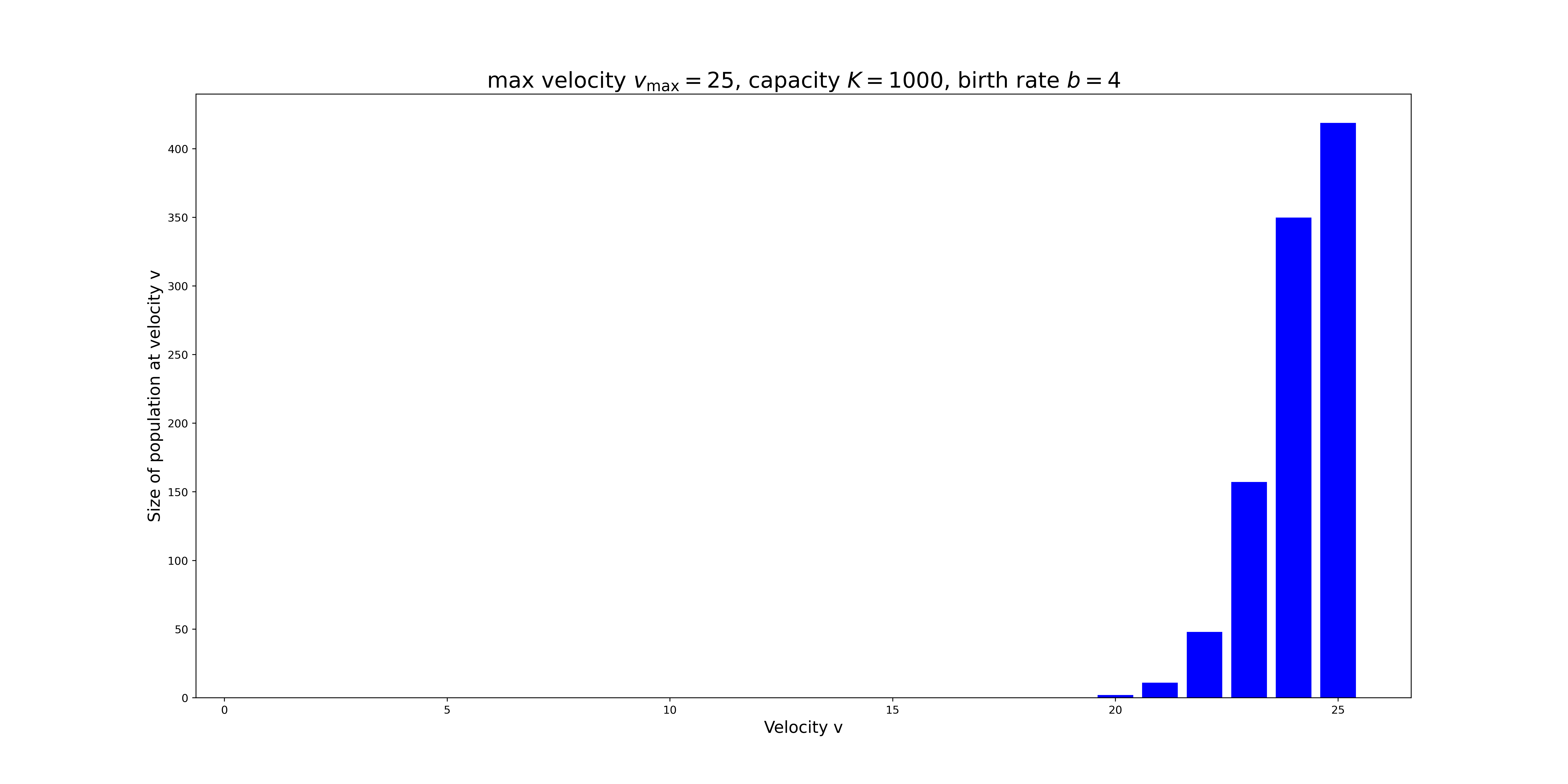}
\caption{Size of population at different velocities after 5000 generations}
\label{fighistogram}
\end{figure}

\begin{figure}
\centering
\includegraphics[width=0.8\textwidth]{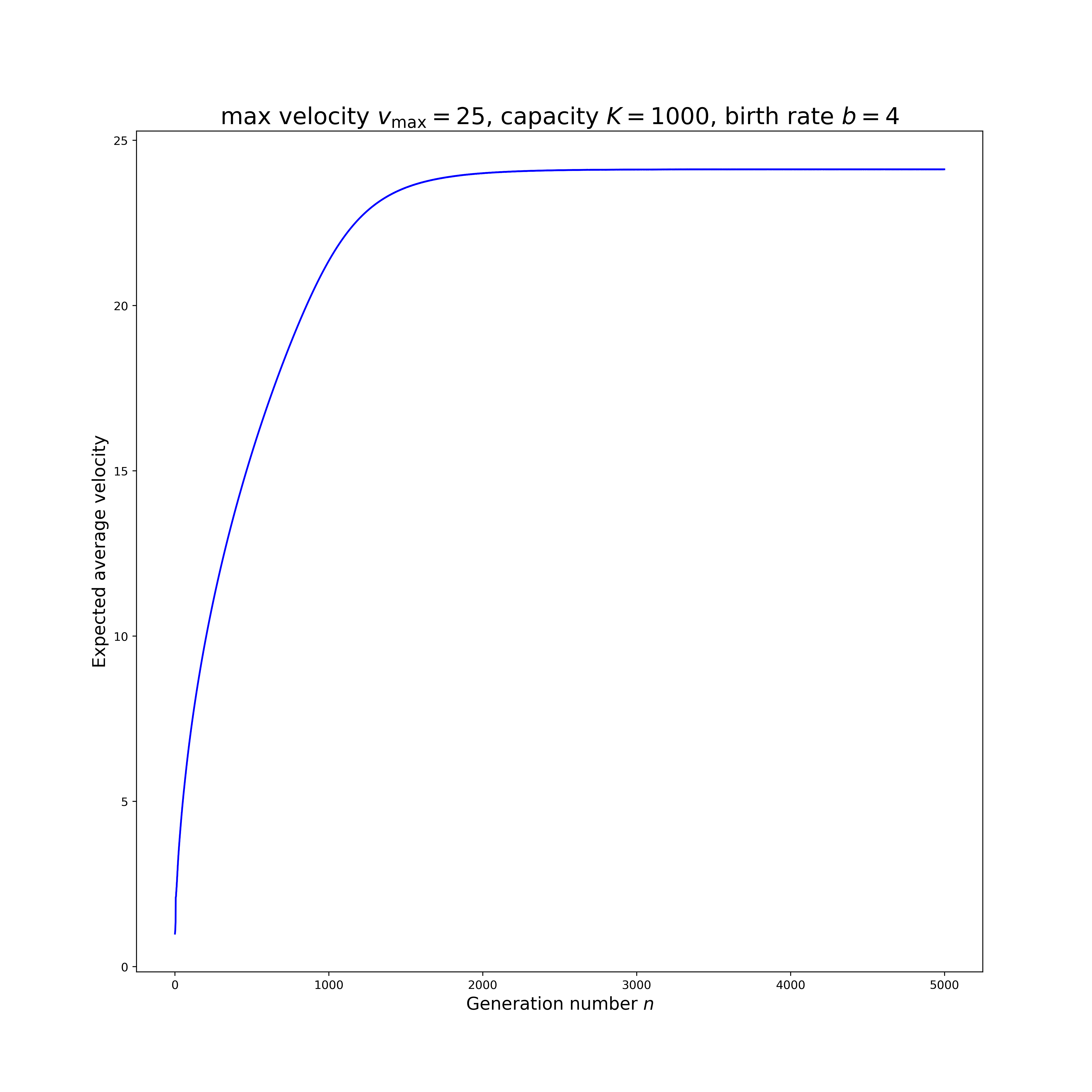}
\caption{Expected average velocity}
\label{expectedaveragevelocity}
\end{figure}

An adjacent piece of evidence supporting Darwinian evolution in our model comes from the expected average velocity of the population, defined as 
$$E_n := \frac{1}{\|P_n\|_1} \sum_{v=1}^{v_{max}}\ vP_n(v).$$

\noindent
In Figure \ref{expectedaveragevelocity}, we see that $E_n$ increases for roughly 2000 generations and then saturates near the maximum possible velocity $v_{max}$. Unfortunately, we are unable to rigorously prove such a behaviour. Hence, in order to stimulate further research, we conclude the current section with the following question: \\

\noindent
\textbf{Open problem.} Assuming that the population does not die off (i.e., $\|P_n\|_1$ does not tend to 0), what is the value of $\liminf_{n\rightarrow \infty}\ E_n?$

\section{Evolution in Markov chains}
\label{Section:Markov}

In the final section of our paper we put forward another model of Darwinian evolution, based on discrete Markov chains (for a thorough exposition of the topic, see \cite{LevinPeresWilmer} or \cite{Norris}). We imagine a population of hares, which can exist in one of $n$ possible states: $s_1, s_2, \ldots, s_n$. We think of these states as representing the average velocity (rounded to one decimal place in order to keep the model discrete) of the population. Thus, a state with low index is interpreted as population with low average velocity, and the opposite is true for a state with high index. 

% transition matrix
As time goes by, the old generation of hares is replaced by their offspring, possibly resulting in a different average velocity of the population. This means a  ``jump'' from state $s_i$ to $s_j$ with some probability $p_{i,j}$. The set of all these probabilities is collected in the transition matrix $P$, which satisfies 
$$\sum_{j=1}^n p_{i,j} = 1$$

\noindent
for every $i=1,\ldots,n$. This is because it is certain that a population either stays in the state it is in, or it jumps to some other state. 

Introduction of the transition matrix facilitates the following description of the model's evolution. Suppose that $\pi = (\pi_i)_{i=1,\ldots,n}$ is a probability distribution, meaning that $\pi_i$ is the probability of population being in state $s_i$. Then, the vector $\pi P$ is the probability distribution after one iteration, i.e., after the old generation of hares has been replaced by a new generation. Let us emphasize that we follow the standard convention in Markov chains, where probability distributions are row (rather than column) vectors. This affects the order of multiplication: probability distributions (row vectors) are on the left, transition matrix $P$ are on the right.   

% interpretation of probabilities
In our model, we want the probabilities $p_{i,j}$ to carry the information regarding both the genetic mutation as well as the external forces exerted on the population. In other words, the transition matrix $P$ should reflect that fact that even if the parent hares run (on average) at velocity $v$, the offspring can run at higher (or lower) velocity. Furthermore, $P$ should also account for the external forces such as a pack of wolves invading the habitat of the population and killing off slower specimens. In that scenario the average velocity should increase, since only the fastest specimens are able to outrun the predators, escaping their fangs.    

% transition matrix - crucial assumption
For the purposes of our model, we make three crucial assumptions regarding the transition matrix:
\begin{description}
	\item[\hspace{0.4cm} (Assumption 1.)] For every $i=1,\ldots,n$ we have $p_{i,i} := 1-\eps$ for some small value of $\eps\in (0, 0.5)$. This reflects the fact that the population is most likely to stay in the same state, i.e., the shift in the average velocity happens on rare occasions. Formally we would say that our Markov chain is aperiodic (see Chapter 1.3 in \cite{LevinPeresWilmer} for a rather technical discussion of aperiodicity).
	
	\item[\hspace{0.4cm} (Assumption 2.)] For every $i=1,\ldots,n-1$ the sequence $p_{i,i}, p_{i,i+1}, \ldots, p_{i,n-1}, p_{i,n}$ is decreasing and $p_{i,i+1} > 0.$ Furthermore, for every $i=2,\ldots,n$ the sequence $p_{i,1}, p_{i,2}, \ldots, p_{i,i-1}, p_{i,i}$ is increasing and $p_{i,i-1} > 0.$ Intuitively, this means that the jump of the population to further states is less likely than jumping to neighbouring ones. Moreover, since it is always possible for the population to jump to the adjacent state, it follows that every state is reachable from any starting position after sufficient number of jumps. Formally, we say that the Markov chain is irreducible.
	
	\item[\hspace{0.4cm} (Assumption 3.)] For every $i=2,\ldots,n-1$ we have 
	\begin{gather}
	\sum_{j < i} p_{i,j} < \sum_{j > i} p_{i,j}.
	\label{drivingforce}
	\end{gather} 

	\noindent
	The inequality means that there is a constant pressure on the population to increase its average velocity. 
\end{description}

Under the assumptions listed above, we expect the population to evolve towards states with higher indices, i.e. higher average velocities. We devote the rest of the section to support this claim by providing ample evidence, though the proof (in general case) is out of our reach. We commence with one of the gems of Markov theory: 

\begin{thm}(see Corollary 1.17 and Theorem 4.9 in \cite{LevinPeresWilmer})\\
If the Markov chain is irreducible then there exists a unique probability distribution $\pistat$ such that $\pistat P = \pistat.$ Moreover, if the Markov chain is aperiodic, then the sequence $(\pi P^k)_{k\in\naturals}$ converges to $\pistat$ for any initial distribution $\pi$.  
\end{thm}

The probability distribution $\pi^*$ is called stationary and, from linear algebra's point of view, its computation amounts to calculating the (left) eigenvector of the eigenvalue $1$ (and normalizing the eigenvector so that the entries sum up to $1$). If our expectations regarding the behaviour of the system are correct, the entries of this eigenvector should be low for low indices and high for high indices. Again, this means that the population will most likely evolve in the direction pointed by the external forces (like a pack of wolves etc.). In order to measure this evolution, we introduce the value
$$E_{\pistat} := \sum_{k=1}^n\ k\pistat_k,$$

\noindent
which we call the expected state. Obviously, it need not be a true state (i.e., an integer between $1$ and $n$) just as the expected value of a die roll (i.e., $3.5$) is not on any of its sides. 

In order to confirm our intuition regarding the behaviour of the expected state, we run several numerical simulations. Our main algorithm (see Algorithm \ref{algorithmMarkov}) uses the following subprograms:
\begin{itemize}
	\item $Uniform(k)$ returns $k$ values drawn from uniform distribution on the unit interval $(0,1),$
	\item $Sum(v)$ returns the sum of all entries in vector $v$, 
	\item $SortAscend(v)$ returns the elements of vector $v$ in ascending order, 
	\item $SortDescend(v)$ returns the elements of vector $v$ in descending order,
	\item $PrincipalEigenvector(P)$ returns the normalized eigenvector with eigenvalue $1$ (i.e., the stationary distribution) of the transition matrix $P$,
\end{itemize}

\noindent
Algorithm \ref{algorithmMarkov} assumes that for every $i=2,\ldots,n-1$ we have
$$\sum_{j < i} p_{i,j} = \delta \eps,\ \hspace{0.4cm}\  \sum_{j > i} p_{i,j} = (1-\delta)\eps$$

\noindent
for a fixed parameter $\delta\in(0,0.5).$ 

\begin{algorithm}
\caption{Stationary distribution $\pi^*$ of transition matrix $P$}\label{algorithmMarkov}
\begin{algorithmic}
\Require $n\in\naturals,\ \eps, \delta \in (0,0.5)$
\For{$i=1,\ldots,n$}
		\State $P(i,i) \gets 1-\eps$
		
		\If{$i==1$}
				\State $v \gets Uniform(n-1)$ 
				\State $v \gets \eps \cdot \frac{v}{Sum(v)}$ 
				\State $v \gets SortDescend(v)$
				\State $P(i, 2\ldots n) \gets v$
				
		\ElsIf{$1<i<n$}
				\State $u \gets Uniform(i-1)$ 
				\State $u \gets \delta \cdot \eps \cdot \frac{u}{Sum(u)}$
				\State $u \gets SortAscend(u)$
		
				\State $v \gets Uniform(n-i)$ 
				\State $v \gets (1-\delta) \cdot \eps \cdot \frac{v}{Sum(v)}$ 
				\State $v \gets SortDescend(v)$
				
				\State $P(i, 1\ldots i-1) \gets u$
				\State $P(i, i+1\ldots n) \gets v$
		
		\ElsIf{$i==n$}
				\State $u \gets Uniform(n-1)$ 
				\State $u \gets \eps \cdot \frac{u}{Sum(u)}$ 
				\State $u \gets SortAscend(u)$
        \State $P(i, 1\ldots n-1) \gets u$
		\EndIf
\EndFor

\State $\pi^* \gets PrincipalEigenvector(P)$
\State \Return $\pi^*$
\end{algorithmic}
\end{algorithm}

\begin{figure}
\centering
\includegraphics[width=\textwidth]{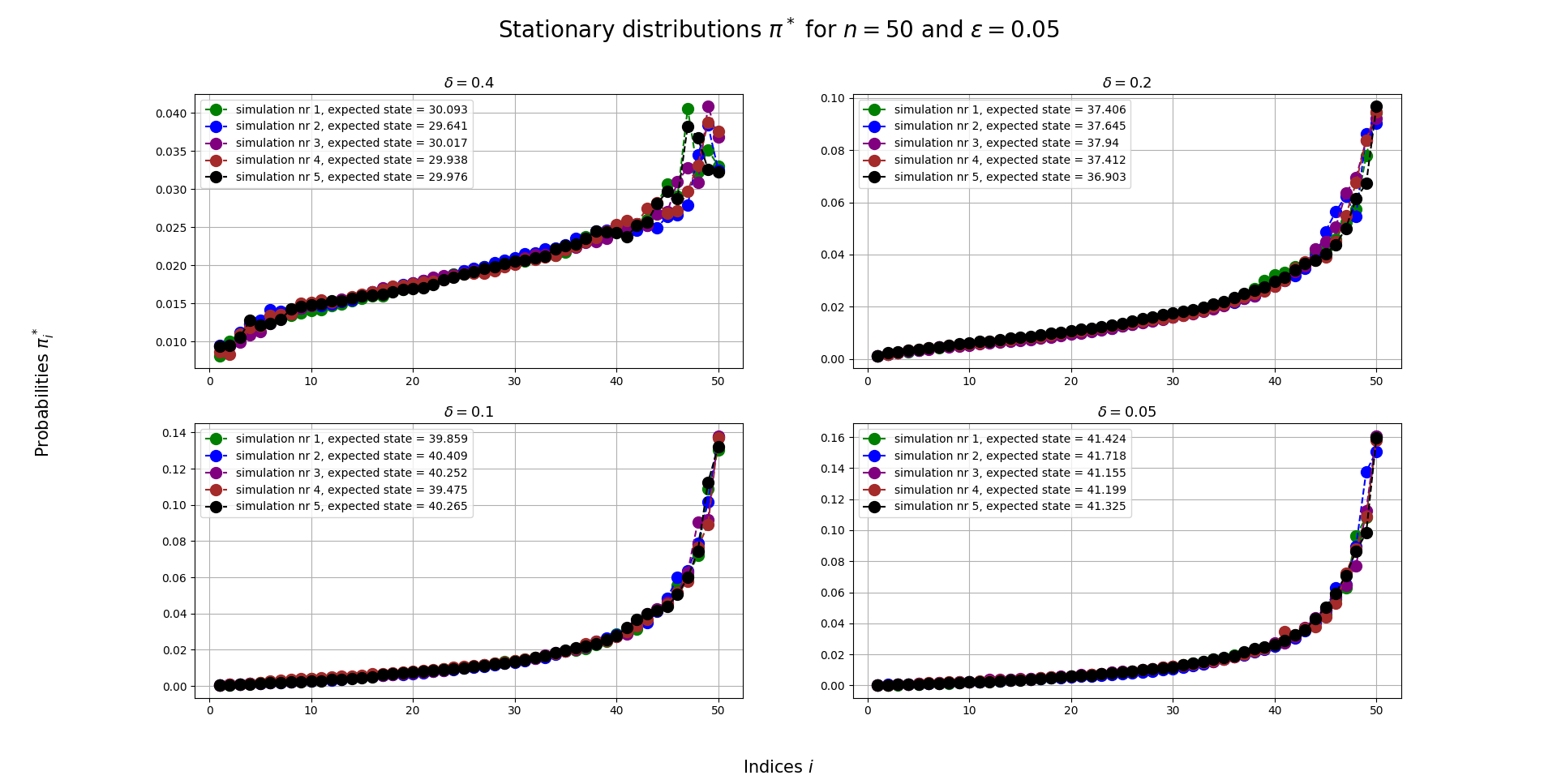}
\caption{Stationary distributions in Markov chains}
\label{Markovsimulation}
\end{figure}

The results of simulations (with $n=50$ states) of Algorithm \ref{algorithmMarkov} are presented in Figure \ref{Markovsimulation}. Every subplot contains 5 simulations performed for a fixed value of $\delta = 0.4, 0.2, 0.1$ or $0.05.$ Even with a naked eye, a common increasing trend is easily detectable throughout all subplots. This tendency amplifies with a decrement of the parameter $\delta$ (corresponding to a stronger environmental pressure towards states with higher indices). The phenomenon is captured by the expected state, which rises from the vicinity of $30$ for $\delta = 0.4$ and exceeds $41$ for $\delta=0.05$. Hence, it is reasonable to suspect that studying the function $\delta \mapsto E_{\pistat}(\delta)$ may lead to a better understanding of Darwinian evolution. Unfortunately, an analysis of the model in full generality eludes our best efforts, so we examine its simplified version:

\begin{thm}
Let $n > 3$ be a natural number and let the transition matrix $P$ be a Hessenberg matrix with $p_{i+1,i} = \delta\eps$ for $i=1,\ldots,n-2$ and $p_{i,i+1} = (1-\delta)\eps$ for $i=2,\ldots,n-1$, i.e.,
\begin{gather}
P(\delta) = \left(\begin{array}{ccccccc}
1-\eps & \eps & 0 & \ldots & \ldots & \ldots & 0\\
\delta\eps & 1-\eps & (1-\delta)\eps & 0 &  &  & \vdots\\
0 & \delta\eps & 1-\eps & (1-\delta)\eps & 0 &  & \vdots \\
\vdots & \ddots & \ddots & \ddots & \ddots & \ddots & \vdots\\
\vdots &  & 0 & \delta\eps & 1-\eps & (1-\delta)\eps & 0\\
\vdots & & & 0 & \delta\eps & 1-\eps & (1-\delta)\eps\\
0 & \ldots & \ldots & \ldots & 0 & \eps & 1-\eps
\end{array}\right).
\label{HessenbergP}
\end{gather}

\noindent 
Then 
\begin{gather}
E_{\pistat}(\delta) = \frac{\delta^{n-2} (2\delta - 1)}{2\delta^{n-1} - 2(1-\delta)^{n-1}} + \frac{\delta^{n-1}(2\delta - 1)}{2(1-\delta)^2 (\delta^{n-1} - (1-\delta)^{n-1})} \cdot \sum_{k=2}^{n-1} k\left(\frac{1-\delta}{\delta}\right)^k + n\frac{(1-\delta)^{n-2} (2\delta - 1)}{2\delta^{n-1} - 2(1-\delta)^{n-1}}
\label{Epistat} 
\end{gather}

\noindent
and, consequently, we have
\begin{gather}
\lim_{\delta \rightarrow 0^+}\ E_{\pistat}(\delta) = \frac{2n-1}{2}.
\label{Epistatlimit}
\end{gather}
\label{Markovchaintheorem}
\end{thm}
\begin{proof}
Since the stationary distribution $\pistat$ satisfies $\pistat(\delta) P(\delta) = \pistat(\delta),$ then comparing the first two components of the equation, we obtain 
\begin{gather*}
\pistat_2(\delta) = \frac{1}{\delta}\pistat_1(\delta) \hspace{0.4cm}\text{ and }\hspace{0.4cm} \pistat_3(\delta) = \frac{1-\delta}{\delta^2}\pistat_1(\delta).
\end{gather*}

\noindent
Carrying out an inductive proof, it is not difficult to establish that 
\begin{gather}
\pistat_k(\delta) = \frac{(1-\delta)^{k-2}}{\delta^{k-1}}\pistat_1(\delta)
\label{pistatk}
\end{gather}

\noindent
for $k=2,3,\ldots,n-1$ and 
\begin{gather}
\pistat_n(\delta) = \left(\frac{1-\delta}{\delta}\right)^{n-2}\pistat_1(\delta).
\label{pistatn}
\end{gather}

Next, we may write the normalization condition 
$$\pistat_1(\delta) + \pistat_2(\delta) + \pistat_3(\delta) + \ldots + \pistat_{n-1}(\delta) + \pistat_n(\delta) = 1$$ 

\noindent
as 
$$\pistat_1(\delta) \left( 1 + \frac{1}{\delta} + \frac{1-\delta}{\delta^2} + \ldots + \frac{(1-\delta)^{n-3}}{\delta^{n-2}} + \left(\frac{1-\delta}{\delta}\right)^{n-2} \right)= 1.$$

\noindent
Applying the formula for the sum of geometric progression and rearranging the terms, we arrive at 
\begin{gather}
\pistat_1(\delta) = \frac{\delta^{n-2} (2\delta - 1)}{2\delta^{n-1} - 2(1-\delta)^{n-1}}.
\label{pistat1}
\end{gather}

\noindent
Combining \eqref{pistatk} and \eqref{pistat1} we have
\begin{gather}
\pistat_k(\delta) = \frac{\delta^{n-1}(2\delta - 1)}{2(1-\delta)^2 (\delta^{n-1} - (1-\delta)^{n-1})} \cdot \left(\frac{1-\delta}{\delta}\right)^k
\label{pistatkfull}
\end{gather}

\noindent
for $k=2,3,\ldots,n-1.$ Next, \eqref{pistatn} and \eqref{pistat1} brought together give 
\begin{gather*}
\pistat_n(\delta) = \frac{(1-\delta)^{n-2} (2\delta - 1)}{2\delta^{n-1} - 2(1-\delta)^{n-1}},
\end{gather*}

\noindent
resulting in the desired form \eqref{Epistat} of $E_{\pistat}.$ 

Finally, since $\lim_{\delta\rightarrow 0^+}\ \pistat_k(\delta) = 0$ for $k=1,\ldots,n-2$ and 
$$\lim_{\delta\rightarrow 0^+}\ \pistat_{n-1}(\delta) = \lim_{\delta\rightarrow 0^+}\ \pistat_n(\delta) = 0.5$$ 

\noindent
we arrive at \eqref{Epistatlimit}. This concludes the proof. 
\end{proof}

Theorem \ref{Markovchaintheorem} is a mathematical testament to Darwinian theory. Once again, it corroborates the claim that species tend to evolve according to external forces acting on them. Much as we are satisfied with such a result, its undeniable downside is the assumption of a very particular form of the transition matrix $P$. Hence, in order to stimulate further research and discussion in the mathematical community, we conclude the paper with the following list of open questions:
\begin{itemize}
	\item What can be said of the expected state $E_{\pistat}$ in the general case, i.e., if the transition matrix $P$ is not necessarily of the form \eqref{HessenbergP}?
	\item Is there any counterpart to $E_{\pistat}$ in the general case or even if $P$ is assumed to be Hessenberg (but not necessarily of the form \eqref{HessenbergP})?
	\item Even if an explicit formula for $E_{\pistat}$ is unattainable, is it possible to establish some generalization of the limit \eqref{Epistatlimit}?
\end{itemize}

\end{document}